\documentclass[12pt, regno]{amsart}
\usepackage[main=english]{babel}
\usepackage[utf8]{inputenc}
\usepackage{amssymb}
\usepackage{amsmath}
\usepackage{amsthm}
\usepackage{euler}
\usepackage{setspace}
\usepackage{thmtools}
\usepackage{mathtools}
\usepackage{bbold}
\usepackage{thm-restate}
\usepackage{etoolbox}
\usepackage{fmtcount}
\usepackage{mathrsfs}
\usepackage{fullpage} 
\usepackage{colonequals} 
\usepackage{enumerate, enumitem}
\usepackage[dvipsnames]{xcolor}
\usepackage{hyperref}
\hypersetup{colorlinks=true, linkcolor=cyan, citecolor=magenta, pdfpagemode=UseNone, pdfstartview={XYZ null null 1.00}}
\usepackage{tikz}
\usetikzlibrary{calc}
\usepackage{cleveref}
\usepackage{indentfirst}

\newlength{\bibitemsep}
\newlength{\bibparskip}
\let\oldthebibliography\thebibliography
\renewcommand\thebibliography[1]{%
  \oldthebibliography{#1}%
  \setlength{\parskip}{\bibitemsep}%
  \setlength{\itemsep}{\bibparskip}%
}

\theoremstyle{plain}
\newcounter{tmp}
\newtheorem{theorem}{Theorem}

\newtheorem{lemma}[theorem]{Lemma}

\newtheorem{corollary}[theorem]{Corollary}

\theoremstyle{remark}
\newtheorem{observation}{Observation}

\theoremstyle{definition}

\newtheorem{definition}[theorem]{Definition}

\DeclareMathOperator{\Z}{\mathbb{Z}} 
\DeclareMathOperator{\F}{\mathbb{F}} 
\def \O {{\mathcal O}} 
\def \P {{\mathbb P}} 

\DeclareMathOperator{\Div}{\mathsf{Div}}

\DeclareMathOperator{\Supp}{\mathsf{Supp}}

\newcommand{\mf}[1]{\mathfrak{#1}} 
\newcommand{\msf}[1]{\mathsf{#1}}  
\newcommand{\nrm}[2]{\mathsf{N}_{#1}(#2)} 

\begin{document}

\title{Global Field Totients}

\author{Santiago Arango-Piñeros}
\address{Department of Mathematics, Emory University,
Atlanta, GA 30322}
\email{santiago.arango@emory.edu}

\author{Juan Diego Rojas}
\address{Departmento de Mathematicas, Universidad de Los Andes,
Bogotá, Colombia}
\email{jd.rojas14@uniandes.edu.co}

\maketitle


\begin{abstract}
    Using a remainder theorem for valuations of a field, we give a new perspective on the norm function of a global field. We define the Euler totient function of a global field and recover the essential analytical properties of the classical arithmetical function, namely the product formula. In addition, we prove the holomorphicity of the associated zeta function. As an application, we recover the analog of the mean value theorem of Erd\H{o}s, Dressler, and Bateman via the Weiner-Ikehara theorem. 
\end{abstract}


\section{Introduction}
\label{sec:intro}
\subsection{Euler's function} The classical arithmetical function of Euler, denoted by $\varphi$, is defined in group theoretic terms as the size of the group of units modulo $n$. The Chinese remainder theorem implies that $\varphi(n)$ is a multiplicative arithmetical function. This means that for every pair of relatively prime positive integers $a$ and $b$, $\varphi(ab) = \varphi(a)\varphi(b)$. Since for every prime number $p$ and every positive integer $r$ the ring $\Z/p^r\Z$ is local with maximal ideal $p\Z/p^r\Z$, the units modulo $p^r$ correspond to the complement of $p\Z/p^r\Z$ inside $\Z/p^r\Z$. This implies $\varphi(p^r) = p^r(1-1/p)$. The unique factorization property of the integers yields the {\it product formula}
\begin{align*}
    \varphi(n) = n\prod_{p\mid n}\left(1-\frac{1}{p}\right) \, .
\end{align*}

The product formula is the soul of several results in classical analytic number theory concerning the Euler function. For instance, the mean value theorem of Erd\H{os}, Dressler, and Bateman, which we proceed to recall. \\

A positive integer $n$ is called a {\it totient} if belongs to the image of $\varphi$. The {\it totient multiplicity} of an integer $n$, denoted by $t(n)$, is the size of the fiber $\varphi^{-1}(n)$. Since $\varphi(n)$ goes to infinity with $n$, $t(n)$ is always finite. The study of the distribution of totient numbers is an alluring research problem studied by many. In particular, the calculation of the mean value of the arithmetical function $t(n)$ was first approached by Erd\H{o}s in \cite{erdos1945some}. Using techniques from probabilistic number theory, he showed the existence of the limit
\begin{align*}
    \lim_{N\to\infty} \, \dfrac{1}{N}\sum_{n=1}^N t(n) = \alpha \, .
\end{align*}
Twenty five years later, Dressler \cite{dressler1970density} explicitly calculated the limit, obtaining the remarkable constant
\begin{align*}
    \alpha = \zeta(2)\zeta(3)/\zeta(6) \,,
\end{align*}
where $\zeta(s)$ is, of course, the Riemann zeta function. Bateman \cite{bateman1972distribution} gave a second proof of this via the Wiener-Ikehara theorem. We summarize this joint effort in the following theorem. 

\begin{theorem}[Erd\H{o}s, Dressler, Bateman] Let $t(n)$ denote the totient multiplicity arithmetical function, and let $\zeta(s)$ denote the Riemann zeta function.  Then, the mean value of $t(n)$ is given by
\begin{align*}
    \lim_{N\to \infty} \dfrac{1}{N} \sum_{n=1}^N t(n) = \zeta(2)\zeta(3)/\zeta(6) \, .
\end{align*}
\label{thm:erdosdresslerbateman}
\end{theorem}

\subsection{Results} The core of this manuscript is a theorem for discrete valuation rings which is reminiscent of the Chinese remainder theorem. The number theoretic results will follow smoothly from this algebraic fact. We refer the reader to Section \ref{sec:dvrrthm} for the complete statement. We where unable to find this result in the literature, so we provide a proof here.

\begingroup
\setcounter{tmp}{\value{theorem}}
\setcounter{theorem}{0} 
\renewcommand\thetheorem{\Alph{theorem}}
\begin{restatable}[DVR remainder theorem]{theorem}{dvrrthm}
    \label{thm:dvrCRT}
    Let $L$ be a field. 
    If $V$ and $W$ are finite disjoint multisets of discrete valuations of $L$, then 
    \begin{align*}
        (R_V \cap R_W)/(I_V \cap I_W) \cong (R_V/I_V) \times (R_W/I_W) \, .
    \end{align*}
\end{restatable}

The main result of this note is the correct definition of the Euler function $\varphi_{K}$ attached to a global field $K$. We claim it is the correct one because the expected analytic properties follow comfortably from it. In particular, we recover the product formula. We refer the reader to Subsection \ref{subsec:notation} to consult the unrevealed notation. 

\begin{restatable}[Product formula]{theorem}{gfprodformula}
    \label{thm:prodformula}
     Let $K$ be a global field with Euler function $\varphi_K$ and norm $\msf{N}_K$. Then, for every nonzero effective divisor $D\in \Div^+(X)$,
     \begin{align*}
         \varphi_K(D) = \nrm{K}{D}\prod_{P\leq D}(1-\nrm{K}{P}^{-1}) \, .
     \end{align*}
\end{restatable}

Following Bateman, we prove the holomorphicity of the {\it totient zeta function} $T_K(s)$ attached to $K$. Furthermore, we show that this function is intimately related to the Weil zeta function $\zeta_K(s)$ of the field.
\begin{restatable}{theorem}{gftotientzeta}
    \label{thm:totientzeta}
    Let $K$ be a global field. The totient zeta function $T_K(s)$ attached to $K$ converges absolutely and uniformly in the open half plane $\sigma > 1$. Furthermore, one has the identity
    \begin{align*}
        \dfrac{T_K(s)}{\zeta_K(s)} =  f_K(s) \colonequals \prod_{P} \left[1 + (\nrm{K}{P}-1)^{-s} - \nrm{K}{P}^{-s}\right]\, , 
    \end{align*}
    where the product ranges over all closed points $P \in X$ and $f_K(s)$ is holomorphic on the open half plane $\sigma > 0$.
\end{restatable}
To conclude, we apply the Weiner-Ikehara theorem to the totient zeta function of $K$ and generalize Theorem \ref{thm:erdosdresslerbateman} to the global field case.

\newpage
\begin{restatable}[Mean value of global field totients]{theorem}{gftotients}
    Let $K$ be a global field, and let $t_K(n)$ denote the totient multiplicity arithmetical function of $K$. Denote by $\rho_K$ the residue of the Weil zeta function $\zeta_K(s)$ at $s=1$. Then, the mean value of $t_K(n)$
    \begin{align*}
        \lim_{N\to \infty} \dfrac{1}{N} \, \sum_{n=1}^N t_K(n) \, = \, \dfrac{\rho_K \, \zeta_K(2)\zeta_K(3)}{\zeta_K(6)} \, .
    \end{align*}
    \label{thm:1}
\end{restatable}
\endgroup
\setcounter{theorem}{\thetmp}

\subsection{Notation}
\label{subsec:notation}
First of all, for a perfect field $k$, a $k$-{\it variety} is a separated scheme of finite type over $k$. We use the letter $q$ to denote a fixed prime power, and the ground field of the varieties considered is always a finite field with $q$ elements $\F_q$. The letter $K$ denotes a {\it global field}, and $X$ denotes the {\it associated scheme}. By this we mean the following:
\begin{itemize}
    \item If $K$ is a {\it number field}, $X$ is the prime spectrum of the ring of integers $\O_K$.
    \item If $K$ is a {\it global function field} over $\F_q$, $X$ is the smooth projective model corresponding to $X$ via resolution of singularities. Recall that $X$ is a smooth integral projective $\F_q$-variety of dimension $1$. 
\end{itemize}
$\Div(X)$ is the group of Weil divisors of $X$ and we denote by $\Div^+(X)$ the semigroup of effective divisors. We restrict the use of $P$ for irreducible divisors. For every irreducible divisor $P$, $\nu_P\colon \Div(X) \to \Z$ is the order function, which extends the discrete valuation of the stalk $\O_{X,P}$. The support of a divisor $D\in\Div(X)$, denoted by $\Supp D$, is the finite set of irreducible divisors $P$ such that $\nu_P(D) \neq 0$. We abbreviate the statement $P\in \Supp D$ by $P \leq D$. 

\begin{definition}[Arithmetic functions on divisors]
\hfill 
\begin{enumerate}
    \item Two effective divisors $D_1, D_2$ are {\it relatively prime} if $(\Supp D_1)\cap(\Supp D_2) = \emptyset$.
    
    \item A function $f\colon\Div^+(X)\to \Z$ is {\it completely multiplicative} if $f(D_1 + D_2) = f(D_1)f(D_2)$ for every pair of divisors $D_1,D_2$.
    
    \item A function $f\colon\Div^+(X)\to \Z$ is {\it multiplicative} if $f(D_1 + D_2) = f(D_1)f(D_2)$ for every pair of relatively prime divisors $D_1,D_2$.
\end{enumerate}
\end{definition}

We use the letter $s$ to denote a complex variable. As is customary in analytic number theory, $\sigma$ and $t$ denote the real and imaginary part of $s$ respectively. Finally, we denote by $\zeta_{K}(s)$ the Weil zeta function of $X$.

\medskip
\section{A remainder theorem for discrete valuations}
\label{sec:dvrrthm}

The motivation for this section comes from the aspiration of unifying the definition of the norm (and Euler function) of a global field. This will become clear after Section \ref{sec:gfef}, although we consider the result to be interesting on its own. We follow the definitions and conventions in Altman and Kleiman's book \cite[Chapter 23]{altman2013term}. In particular, all discrete valuations are surjective.

\medskip
In this section $L$ will denote a field. For a discrete valuation $v$ of $L$, $R_v \colonequals \{x \in L \colon v(x) \geq 0\}$ denotes the corresponding discrete valuation ring, and $\mf{m}_v \colonequals \{x \in L \colon v(x) > 0\}$ denotes its maximal ideal. For every $v$, we fix a uniformizer $t_v$ at $v$.

\medskip
A multiset is a pair $(A,\mu)$ where $A$ is a set and $\mu\colon A\to \Z^+$ is a function, called the multiplicity function of $A$. We omit the multiplicity function from the notation and use $\mu$ to denote the multiplicity function of all multisets. We consider multisets of discrete valuations of the field $L$ and attach certain subrings and ideals to them.

\begin{definition}
    \label{def:multisetRing}
    Let $L$ be a field and let $V$ be a  multiset of discrete valuations of $L$. 
    \begin{enumerate}
        \item Define the {\it ring associated to} $V$ by
        \begin{align*}
            R_V \colonequals \bigcap_{v\in V} R_v \, .
        \end{align*}
        
        \item Define the {\it ideal associated to} $V$ by
        \begin{align*}
            I_V \colonequals \bigcap_{v\in V} \mf{m}_v^{\mu(v)} \, .
        \end{align*}
    \end{enumerate}
\end{definition}

\dvrrthm*

\begin{proof}
    The composition of the inclusion with the natural projection gives morphisms from $R_V \cap R_W$ to the quotients $R_V/I_V$ and $R_W/I_W$. Therefore, the map that takes an element in the intersection $f\in R_V\cap R_W$ and assigns the ordered pair $(f+I_V, f + I_W)$ is a ring homomorphism. Furthermore, its kernel is precisely $I_V\cap I_W$. In order to complete the proof, we need to show that this homomorphism is surjective. Take an element $(f_1 + I_V, f_2 + I_W)$, and define the positive integers 
    \begin{align*}
        N \colonequals \max_{v\in V}\{\mu(v), -v(f_2)\} \, \quad M \colonequals \max_{w\in W}\{\mu(w), -w(f_1)\} \, .
    \end{align*}
    Define $t_V \colonequals \prod_{v\in V} t_v$ and $t_W \colonequals \prod_{w\in W} t_w$ and let $f\colonequals f_1t_W^M +f_2t_V^N \in L$. Then for every $v\in V$ and $w\in  W$ we have the inequalities
    \begin{align*}
        \nu(f) \geq \min\{\nu(f_1), \nu(f_2) + N\} \geq 0 \, , \quad  w(f) \geq \min\{w(f_1) + M, w(f_2)\} \geq 0 \, .
    \end{align*}
    This implies that $f\in R_\nu\cap R_w$. By construction, $f + I_V = f_1 + I_V$ and $f + I_W = f_2 + I_W$. 
\end{proof}

\medskip
\section{The global field Euler function}
\label{sec:gfef}

The norm of a closed point in $X$ is the dimension of the residue field $k(P)$. One then defines the norm of an effective divisor by extending multiplicatively. When $K$ is a number field and $I$ is an integral ideal, the norm of $I$ coincides with the cardinality of the finite ring $\O_K/I$. When $K$ is a global function field over $\F_q$, the norm of an effective divisor $D$ is $q^{\deg D}$. In this section we introduce a new perspective that unifies the proofs of our results. 

\medskip
Note that every divisor $D\in \Div(X)$ corresponds to a finite multiset of closed points in $X$, where the multiplicity of each closed point $P$ is $\nu_P(D)$. Since $X$ is smooth, each closed point corresponds to a discrete valuation of $K$. Following Definition \ref{def:multisetRing}, for an effective divisor $D\in \Div^+(X)$, we denote by $\O_D$ the ring associated to the multiset\footnote{This approach was suggested to us by Guillermo Mantilla-Soler.} of valuations given by $D$. Similarly, $I_D$ denotes the ideal associated to the multiset of valuations given by $D$. Explicitly,
\begin{align*}
        \O_D \colonequals \bigcap_{P\in\, \Supp D} \O_{X,P} \, , \quad
        I_D \colonequals \bigcap_{P\in\, \Supp D} \mf{m}_P^{\nu_P(D)} \, .
    \end{align*}
\begin{definition}[Numerical norm]
\label{def:norm}
    Let $K$ be a global field. The {\it norm} of $K$ is the function $\msf{N}_K\colon \Div^+(X) \to \Z$ given by
    \[
    \msf{N}_K(D) \colonequals \begin{cases}
    1, & \text{ for }  D=0, \\
    \#(\O_D/I_D), & \text{ otherwise.}
    \end{cases}
    \]
\end{definition}

Applying Theorem \ref{thm:dvrCRT} to a global field $K$, we show that the perspective introduced in Definition \ref{def:norm} coincides with the usual definition.

\begin{corollary}[Properties of the norm]
    \label{cor:main}
    Let $K$ be a global field. Then, for every nonzero effective divisor $D \in \Div^+(X)$, we have the ring isomorphism
    \begin{align*}
        \O_D/I_D \cong \prod_{P \leq D} \O_{X,P}/\mf{m}_P^{\nu_P(D)} \, .
    \end{align*}Moreover,
    \begin{enumerate}
        \item[a)] When $K$ is a number field, and $I$ is the ideal $I \colonequals\prod_{P \leq D}P^{\nu_P(D)}$ the rings $\O_D/I_D$ and $\O_K/I$ are isomorphic.  In particular, $\nrm{K}{D} = \#(\O_K/I)$.
        \item[b)] When $K$ is a global function field over $\F_q$, $\nrm{K}{D} = q^{\deg D}$. 
    \end{enumerate}
    In particular, $\msf{N}_K \colon \Div^+(X) \to \Z$ is a completely multiplicative function on divisors.
\end{corollary}

\begin{proof}
    The isomorphism follows from Theorem \ref{thm:dvrCRT} to this context. Item (a) follows from the isomorphism
    \begin{align*}
        \O_D/I_P^n = \O_D/\mf{m}_P^n \cong (\O_K)_P/P^n(\O_K)_P \cong \O_K/P^n \, ,
    \end{align*}
    for every maximal ideal $P\lhd \O_K$, and the Chinese remainder theorem. Item (b) follows from the observation that for a closed point $P\in X$, $\O_D = \O_{X,P}$ and $I_P = \mf{m}_P$, so that $\O_P/I_P$ is precisely the residue field of $P$.
\end{proof}

The author's first guess for the definition of the Euler function of a global field was the same from elementary number theory. In other words, we assumed that $\varphi_K(D)$ should count the number of effective divisors with norm not exceeding $\nrm{K}{D}$, which are coprime to $D$. Even though this approach works in the number field case, the reader may check that this naive Euler function is not even multiplicative in the case of $\P^1$. Inspired by the group theoretic definition from Section \ref{sec:intro}, Definition \ref{def:norm}, and Corollary \ref{cor:main}, we propose the following definition.

\begin{definition}[Global field Euler function]
    \label{def:gfeuler}
    Let $K$ be a global field. The {\it Euler function} of $K$ is the function $\varphi_K\colon \Div^+(X) \to \Z$ given by
    \[
    \varphi_K(D) \colonequals \begin{cases}
    1, & \text{ for }  D=0, \\
    \#(\O_D/I_D)^\times, & \text{ otherwise.}
    \end{cases}
    \]
\end{definition}

In particular, for every closed point $P\in X$, $\varphi_K(P) = \nrm{K}{P} -1$. By Corollary \ref{cor:main}, $\varphi_K$ is multiplicative. The product formula will follow from the calculation of the units in the ring $\O_{rP}/I_{rP}$ for $P\in X$ closed and $r\geq 1$.

\gfprodformula*

\begin{proof}
    Take a closed point $P\in X$, and let $r\geq 1$. Consider the divisor $D = rP$, and note that $\O_D = \O_{X,P}$ and $I_D = \mf{m}_P^r$. The ring $\O_D/I_D = \O_{X,P}/\mf{m}_P^r$ is local with maximal ideal $\mf{m}_P/\mf{m}_P^r$. Therefore, $(\O_{X,P}/\mf{m}_P^r)^\times = \O_{X,P}/\mf{m}_P^r - \mf{m}_P/\mf{m}_P^r$. From the exact sequence
    \begin{align*}
        0 \longrightarrow \mf{m}_P/\mf{m}_P^r \longrightarrow \O_{X,P}/\mf{m}_P^r \longrightarrow \O_{X,P}/\mf{m}_P \longrightarrow 0 \,,
    \end{align*}
    we deduce that $\#(\mf{m}_P/\mf{m}_P^r) = \nrm{K}{P}^{r-1}$, and $$\varphi_K(rP) = \nrm{K}{P}^r - \nrm{K}{P}^{r-1} = \nrm{K}{rP}(1-\nrm{K}{P}^{-1}).$$
    The result follows for general divisors from this calculation and Corollary \ref{cor:main}.
\end{proof}

\medskip
\section{Holomorphicity of the totient zeta function}
Just as the Weil zeta function of $K$ is the series associated to the norm $\msf{N}_K$
\begin{align*}
    \zeta_K(s) = \sum_{D\in\Div^+(X)} \nrm{K}{D}^{-s} \,, \quad \text{ when } \sigma > 1 \,,
\end{align*}
it is natural to consider the series associated to the Euler function $\varphi_K$. Following Bateman, we define the {\it totient zeta function} as the formal series
\begin{align*}
    T_K(s) \colonequals \sum_{D\in\Div^+(X)} \varphi_K(D)^{-s} \,.
\end{align*}
In this section we show that $T_K(s)$ converges absolutely and uniformly on $\sigma > 1$. This will follow from the convergence of $\zeta_K(s)$ and the convergence of an Euler product that appears spontaneously as a factor of $T_K(s)$.

\begin{lemma}
    Let $K$ be a global field. For each closed point $P\in X$, consider the function $F_P(s) \colonequals 1 + (\nrm{K}{P}-1)^{-s} - \nrm{K}{P}^{-s}$. The Euler product $\prod_P F_P(s)$ converges uniformly to an holomorphic function $f_K(s)$ on the open half plane $\sigma > 0$.
    \label{lemma:fprod}
\end{lemma}
\begin{proof}
     We drop the subindex $K$ in the norm and Euler function for the sake of clarity. Fix $\delta > 0$ and let $U_\delta$ be the open half plane $\sigma>\delta$. We have the trivial lower bound $\nrm{}{D}\geq 2$ for every divisor nonzero effective divisor $D$. By choosing the principal branch of the logarithm, $F_P(s)$ is holomorphic in $U_\delta$. Thus $\left(F_P(s)\right)_P$ is a sequence of holomorphic functions on $U_\delta$. Fix $P$ and denote momentarily $x = \nrm{}{P}$. Then
     \begin{equation}
        | F_P(s) - 1 | = | (x-1)^{-s} - x^{-s} | = \Big| s\int_{x-1}^x z^{-s-1} \, \mathrm{d}z \Big| \leq |s|(x-1)^{-\sigma-1} \, . \label{ineq:FP}
     \end{equation}
    By Inequality $\ref{ineq:FP}$, we have that for every $P$,
    \begin{align*}
        |F_P(s) - 1| \ll \nrm{}{P}^{-\delta-1} \, ,
    \end{align*}
    uniformly on every compact subset of $U_\delta$. Since
    \begin{align*}
        \sum_P |F_P(s) -1| \ll \sum_P \nrm{}{P}^{-\delta-1} \leq \zeta_K(\delta + 1) < \infty\, , 
    \end{align*}
    we conclude that the infinite product $\prod_P F_P$ converges locally uniformly on $U_\delta$, for every $\delta>0$,  and therefore represents an holomorphic function $f_K(s)$ on the half plane $\sigma > 0$.
\end{proof}
The multiplicativity of $\varphi_K$ allows us to manipulate the formal series $T_K(s)$ to get a familiar function, of which we know the legitimacy of its convergence. From this informal idea we get the following theorem.

\gftotientzeta*

\begin{proof}
We drop the subindex $K$ in the norm and Euler function for the sake of clarity. Consider the following manipulation of formal series and products.

\begin{align*}
    T_K(s) &\colonequals \,  \sum_{D\geq 0} \varphi(D)^{-s} \\
    &= \, \prod_{P} \sum_{m=0}^\infty \varphi(mP)^{-s} \tag{By the multiplicativity of $\varphi$}\\
    &= \, \prod_{P}\Big(1+ \sum_{m=1}^\infty \Big[\nrm{}{P}^{m-1}(\nrm{}{P}-1)\Big]^{-s} \Big) \tag{By Theorem \ref{thm:prodformula}}. \\
    \intertext{Observe that the product converges for $\sigma > 1$, and therefore we have absolute convergence of $T_K(s)$ in this half plane.}
    &= \, \prod_{P}\Big(1+ (\nrm{}{P}-1)^{-s}\sum_{n=0}^\infty \nrm{}{P}^{-ns}\Big) \tag{Factor $(\msf{N}(P)-1)^{-s}$ and put $n=m-1$}\\
    &= \, \prod_{P}\Big(1+ (\nrm{}{P}-1)^{-s}(1-\nrm{}{P}^{-s})^{-1}\Big) \tag{Geometric series}\\
    &= \zeta_K(s) \prod_{P}\Big(1 + (\nrm{}{P}-1)^{-s} - \nrm{}{P}^{-s}\Big) \tag{Factoring the zeta function} \\
    &= \zeta_K(s) f_K(s) \, .
\end{align*}
By Lemma \ref{lemma:fprod}, we have that uniform convergence holds in $\sigma > 1$. Therefore, the rearrangements made are all justified in this domain, yielding the desired result. In particular, rewriting $T_K(s)$ as a Dirichlet series we get that the totient multiplicity arithmetical function $t_K(n)$ is well defined.
\end{proof}

\bigskip
\section{Mean value of global field totients}
In this section we replicate Bateman's argument to calculate the mean value of the totient multiplicity arithmetical function $t_\msf{K}(n)$ attached to the global field $K$. The proof of this result relies heavily on the holomorphicity of $T_K(s)$ and the Wiener-Ikehara theorem. The original references are \cite{ikehara1931extension} and \cite{wiener1932tauberian}. Also, the book \cite[Section 8.3]{montgomery2007multiplicative} contains a complete discussion of this result. We recall the statement here for convenience.

\begin{theorem}[Wiener-Ikehara]
    Let $(a_n)_{n=1}^\infty$ be a sequence of non negative real numbers such that the associated Dirichlet series $\sum a_n n^{-s}$ converges in the open half plane $\sigma > 1$. If there exists a constant $\alpha$ and a continuous function $h(s)$ on the closed half plane $\sigma \geq 1$ such that
    $$ h(s) = \left(\sum_{n=1}^\infty \frac{a_n}{n^s}\right) - \left(\frac{\alpha}{s-1}\right)\, $$
    on $\sigma > 1$, then
    $$ \lim_{N\to \infty} \, \frac{1}{N} \, \sum_{n=1}^N a_n \, = \, \alpha \, . $$
    \label{thm:wiener-ikehara}
\end{theorem}

\begin{observation}
    \label{obs:WIthm}
    In particular, the hypothesis of the theorem are satisfied when $\sum a_n n^{-s}$ is meromorphic in $\sigma \geq 1$ with a unique (simple) pole at $s=1$ of residue $\alpha$. 
\end{observation}

\gftotients*

\begin{proof}
By Theorem \ref{thm:totientzeta}, we have the equality
\begin{align*}
    \sum_{m=1}^\infty \frac{t_\msf{K}(m)}{m^s} \, = \, \zeta_\msf{K}(s)f_\msf{K}(s) \, .
\end{align*}
A direct calculation gives $f_\msf{K}(1) = \zeta_\msf{K}(2)\zeta_\msf{K}(3)/\zeta_\msf{K}(6)$. Since $f_\msf{K}(s)$ is holomorphic on $\sigma>0$ and $\zeta_\msf{K}(s)$ has a unique (simple) pole at $s=1$, we have that $T_\msf{K}(s)$ is meromorphic on $\sigma\geq 1$ and has a unique (simple) pole at $s=1$ with residue
\begin{align*}
    \mathrm{Res}_{s=1}\Big[\, T_\msf{K}(s)\,\Big] = \lim_{s\to 1} (s-1)T_\msf{K}(s) = \lim_{s\to 1} (s-1)\zeta_\msf{K}(s)f_\msf{K}(s) = \rho_\msf{K} f_\msf{K}(1) \, .
\end{align*}

As we remarked in Observation \ref{obs:WIthm}, the hypothesis of Weiner-Ikehara theorem are satisfied in this case, so we conclude that 
\begin{align*}
    \lim_{N\to \infty} \dfrac{1}{N} \, \sum_{m=1}^N t_\msf{K}(m) \, = \dfrac{\rho_\msf{K} \, \zeta_\msf{K}(2)\zeta_\msf{K}(3)}{\zeta_\msf{K}(6)} \, .
\end{align*}
\end{proof}


\bibliography{refs}{}
\bibliographystyle{amsalpha}


\end{document}